\documentclass{amsproc}

\usepackage{amsmath,amsfonts,amsthm,mathrsfs}
\usepackage{amssymb}

\usepackage[alphabetic,initials]{amsrefs}

\usepackage[unicode,bookmarks]{hyperref}
\usepackage[usenames,dvipsnames]{xcolor}
\hypersetup{colorlinks=true,citecolor=NavyBlue,linkcolor=BrickRed,urlcolor=Orange}

\usepackage{enumitem}

\newenvironment{aenumerate}{%
	\begin{enumerate}[label=(\alph{*}), ref=(\alph{*})]
}{%
	\end{enumerate}%
}

\usepackage{tikz-cd}
\tikzset{commutative diagrams/arrow style=math font}

\newcommand{\MHM}{\operatorname{MHM}}
\newcommand{\Dmod}{\mathscr{D}}
\newcommand{\Mmod}{\mathcal{M}}
\newcommand{\Nmod}{\mathcal{N}}
\newcommand{\shT}{\mathscr{T}}

\newcommand{\derL}{\mathbf{L}}

\newcommand{\decal}[1]{\lbrack #1 \rbrack}
\newcommand{\shH}{\mathcal{H}}
\newcommand{\tensor}{\otimes}
\newcommand{\NN}{\mathbb{N}}
\newcommand{\ZZ}{\mathbb{Z}}
\newcommand{\QQ}{\mathbb{Q}}
\newcommand{\CC}{\mathbb{C}}
\newcommand{\menge}[2]{\bigl\{ \thinspace #1 \thinspace\thinspace \big\vert%
\thinspace\thinspace #2 \thinspace \bigr\}}
\newcommand{\Menge}[2]{\Bigl\{ \thinspace #1 \thinspace\thinspace \Big\vert%
\thinspace\thinspace #2 \thinspace \Bigr\}}

\DeclareMathOperator{\Supp}{Supp}
\DeclareMathOperator{\Res}{Res}
\DeclareMathOperator{\Ch}{Ch}
\newcommand{\define}[1]{\emph{#1}}
\newcommand{\lie}[2]{\lbrack #1, #2 \rbrack}
\newcommand{\shf}[1]{\mathscr{#1}}
\newcommand{\OX}{\shf{O}_X}
\newcommand{\OmX}{\Omega_X}
\newcommand{\omX}{\omega_X}
\newcommand{\restr}[1]{\big\vert_{#1}}

\newcommand{\fl}{f_{\ast}}
\newcommand{\jl}{j_{\ast}}
\newcommand{\ju}{j^{\ast}}
\newcommand{\ius}{i^{!}}
\newcommand{\il}{i_{\ast}}
\newcommand{\iu}{i^{\ast}}
\newcommand{\into}{\hookrightarrow}

\newcommand{\shF}{\shf{F}}
\newcommand{\shO}{\shf{O}}
\newcommand{\theoremref}[1]{\hyperref[#1]{Theorem~\ref*{#1}}}
\newcommand{\lemmaref}[1]{\hyperref[#1]{Lemma~\ref*{#1}}}
\newcommand{\definitionref}[1]{\hyperref[#1]{Definition~\ref*{#1}}}
\newcommand{\propositionref}[1]{\hyperref[#1]{Proposition~\ref*{#1}}}
\newcommand{\conjectureref}[1]{\hyperref[#1]{Conjecture~\ref*{#1}}}
\newcommand{\corollaryref}[1]{\hyperref[#1]{Corollary~\ref*{#1}}}
\newcommand{\exampleref}[1]{\hyperref[#1]{Example~\ref*{#1}}}
\newcommand{\Db}{\mathrm{D}^b}
\newcommand{\Dbcoh}{\mathrm{D}_{\mathit{coh}}^b}
\newcommand{\OZ}{\shO_Z}
\newcommand{\omXY}{\omega_{X/Y}}
\newcommand{\omXYnu}{\omXY^{\nu}}
\newcommand{\Sh}[1]{\hat{S}^{#1}}
\newcommand{\Hev}{H_{\mathit{ev}}}
\newcommand{\shHev}{\shH_{\mathit{ev}}}
\newcommand{\famY}{\mathscr{Y}}
\newcommand{\xit}{\xi'}
\newcommand{\Mt}{M^{\boxtimes m}}
\newcommand{\Mmodt}{\Mmod^{\boxtimes m}}
\newcommand{\QHX}{\QQ_X^H}

\newtheorem{theorem}{Theorem}[section]
\newtheorem{lemma}[theorem]{Lemma}

\theoremstyle{definition}
\newtheorem{definition}[theorem]{Definition}
\newtheorem{example}[theorem]{Example}

\theoremstyle{remark}

\numberwithin{equation}{section}

\begin{document}

\title{Weak positivity via mixed Hodge modules}

\author{Christian Schnell}
\address{Department of Mathematics, Stony Brook University,
Stony Brook, NY 11794, USA}
\email{\tt cschnell@math.sunysb.edu}
\thanks{}

\dedicatory{Dedicated to Herb Clemens}

\subjclass[2010]{Primary 14E30; Secondary 14F10}
\keywords{Weak positivity; Mixed Hodge module}

\date{}

\begin{abstract}
We prove that the lowest nonzero piece in the Hodge filtration of a mixed Hodge
module is always weakly positive in the sense of Viehweg.
\end{abstract}

\maketitle

\section*{Foreword}

{\small ``Once upon a time, there was a group of seven or eight beginning graduate
students who decided they should learn algebraic geometry. They were going to do it
the traditional way, by reading Hartshorne's book and solving some of the exercises
there; but one of them, who was a bit more experienced than the others, said to his
friends: `I heard that a famous professor in algebraic geometry is coming here soon;
why don't we ask him for advice.' Well, the famous professor turned out to be a very
nice person, and offered to help them with their reading course. In the end, four out
of the seven became his graduate students \dots and they are very grateful for the
time that the famous professor spent with them!''}

\section{Introduction}

\subsection{Weak positivity}

The purpose of this article is to give a short proof of the Hodge-theoretic part of
Viehweg's weak positivity theorem. 

\begin{theorem}[Viehweg] \label{thm:Viehweg}
Let $f \colon X \to Y$ be an algebraic fiber space, meaning a surjective morphism
with connected fibers between two smooth complex projective algebraic varieties. Then
for any $\nu \in \NN$, the sheaf $\fl \omXYnu$ is weakly positive.
\end{theorem}

The notion of weak positivity was introduced by Viehweg, as a kind of birational
version of being nef. We begin by recalling the -- somewhat cumbersome --
definition. Let $\shF$ be a torsion-free coherent sheaf on a smooth algebraic variety
$X$.  We shall denote by $S^n(\shF)$ the $n$-th symmetric power of $\shF$, and by
$\Sh{n}(\shF)$ its reflexive hull. A more concrete definition is that 
\[
	\Sh{n}(\shF) = \jl S^n(\ju \shF),
\]
where $j$ is the inclusion of the maximal open subset over which $\shF$ is
locally free. The formula holds because the complement has codimension at least two
in $X$.

\begin{definition}
A torsion-free coherent sheaf $\shF$ is \define{weakly positive} on an open subset $U
\subseteq X$ if, for every ample line bundle $L$ on $X$, and for every $m \in \NN$, the
restriction morphism
\[
	H^0 \bigl( X, \Sh{mn}(\shF) \tensor L^{\tensor n} \bigr)
		\to \Sh{mn}(\shF) \tensor L^{\tensor n} \restr{U}
\]
is surjective for all sufficiently large $n \in \NN$.
\end{definition}

If we think of $\Sh{mn}(\shF) \tensor L^{\tensor n}$ as being the $mn$-th
symmetric power of the (non-existent) object $\shF \tensor L^{\tensor 1/m}$, then the
definition means the following: after tensoring $\shF$ by an arbitrarily small
fraction of an ample line bundle, all sufficiently large symmetric powers are
generated over $U$ by their global sections.

\begin{example}
Weak positivity on $X$ is equivalent to nefness. For line bundles, being weakly 
positive is the same thing as being pseudo-effective.
\end{example}

The most notable application of the weak positivity theorem is Viehweg's proof of the
Iitaka conjecture over a base of general type. Let $f \colon X \to Y$ be an algebraic
fiber space, with general fiber $F$; Iitaka's conjecture predicts that
\[
	\kappa(X) \geq \kappa(Y) + \kappa(F).
\]
Viehweg proved the conjecture when $Y$ is of general type. Roughly speaking, he uses
the weak positivity of $\fl \omXYnu$, together with the fact that $\omega_Y$ is big,
to produce sufficiently many global sections of $\omX^{\nu}$. They survey article
\cite{Viehweg} contains a nicely written account of these matters.

\subsection{Related results}

The idea of using methods from Hodge theory to prove the positivity of certain
sheaves goes back at least to Fujita and Kawamata. To put Viehweg's theorem in
context, let me briefly mention a few other positivity theorems for the direct image
of $\omXY$ and its powers:

\begin{enumerate}
\item Kawamata \cite{Kawamata} proved that if the fiber space $f \colon X \to Y$ is
``nice'', then $\fl \omXY$ is locally free and nef. Nice means, roughly speaking,
that the singularities of $f$ should occur over a normal crossing divisor, and that
the local monodromy should be unipotent. The proof uses Hodge theory.
\item Viehweg deduced from Kawamata's result that $\fl \omXY$ is always weakly
positive; by studying certain well-chosen branched coverings of $X$, he obtained the
same result for $\fl \omXYnu$ with $\nu \geq 2$.
\item Koll\'ar gave a simpler proof of Viehweg's theorem using his vanishing theorem
for higher direct images of dualizing sheaves.
\item Earlier this year, Fujino, Fujisawa, and Saito \cite{FujinoFujisawaSaito}
proved a general nefness theorem for mixed Hodge modules that has Kawamata's theorem
as a special case.
\end{enumerate}

There are many other results of this type; the list above contains only those
that are most closely related to the topic of this paper.

\subsection{Main theorem}

Let $\shH$ be a polarizable variation of Hodge structure on a the complement of a
divisor $D$ in a smooth projective algebraic variety $X$, and let $p \in \ZZ$ be the
largest integer for which the Hodge bundle $F^p \shH$ is nontrivial. The general
philosophy that emerges from Kawamata's work \cite{Kawamata} is that $F^p \shH$
extends in a canonical way to a nef vector bundle on $X$, provided that $D$ is a
normal crossing divisor and $\shH$ has unipotent local monodromy. (Without these
assumptions, there are examples \cite{FujinoFujisawaSaito}*{Remark~4.6(v)} where
nefness fails.) The analogy with Viehweg's theorem suggests that, without these
assumptions, there should still be a canonical extension of $F^p \shH$ that is weakly
positive. 

This expectation turns out to be correct, if one takes for the extension of $F^p
\shH$ the one given by Saito's theory of mixed Hodge modules. Here is the
result, proved in collaboration with Mihnea Popa.

\begin{theorem} \label{thm:main}
Let $M$ be a mixed Hodge module on a smooth complex projective variety $X$. If the
underlying filtered $\Dmod$-module satisfies $F_{k-1} \Mmod = 0$, then the coherent
sheaf $F_k \Mmod$ is weakly positive on the open subset where $M$ is a variation
of mixed Hodge structure.
\end{theorem}

Viehweg's theorem (in the case $\nu = 1$) is an immediate consequence. Given a
morphism $f \colon X \to Y$ as in \theoremref{thm:Viehweg}, we let $M$ be the direct
image of the trivial Hodge module on $X$; to be precise, $M = H^0 \fl \QHX
\decal{\dim X}$, in Saito's notation. Setting $k = \dim Y - \dim X$, it can be shown
that 
\[
	F_{k-1} \Mmod = 0 \quad \text{and} \quad F_k \Mmod \simeq \fl \omXY,
\]
and \theoremref{thm:main} implies that this sheaf is weakly positive on the smooth
locus of $f$. The weak positivity of $\fl \omXYnu$ for $\nu \geq 2$ follows as in
Viehweg's original argument by considering certain branched coverings of $X$.

\section{Mixed Hodge modules}

Although it looks more general, \theoremref{thm:main} is not really a new result,
because it could  be deduced from Kawamata's theorem. The point of presenting it is
that mixed Hodge modules appear to be the natural setting: as we will see towards the
end of the talk, the proof of the theorem is extremely short. This may be only of
academic interest in this case, where we are basically reproving an existing result
-- but in other situations, the use of mixed Hodge modules may allow us to go much
further than existing methods.\footnote{At this point in the lecture, Wilfried Schmid
interjected, ``I think it's fair to say that mixed Hodge modules have really been
greatly underused.''} 

Before explaining the proof of \theoremref{thm:main}, I would like to say a few words
about mixed Hodge modules and their applications to algebraic geometry; I will also
try to motivate their use with a specific example. To simplify the discussion, let me
concentrate on the case of pure Hodge modules.

\subsection{Variations of Hodge structure}

Hodge modules are a generalization of variations of Hodge structure. For applications
to algebraic geometry, the essential features of a variation of Hodge structure $H$
are the following:
\begin{enumerate}
\item A holomorphic vector bundle $\shH$ with a flat connection $\nabla \colon \shH
\to \OmX^1 \tensor \shH$;
\item a Hodge filtration $F^{\bullet} \shH$ by holomorphic subbundles, subject to the
relation
\[
	\nabla(F^p \shH) \subseteq \OmX^1 \tensor F^{p-1} \shH,
\]
called Griffith transversality (by everyone except Griffiths, who calls it
the infinitesimal period relation).
\end{enumerate}

In practice, there are various pieces of additional data: a rational structure; a
polarization; a weight filtration (in the mixed case); etc. But the vector bundles
$\shH$ and $F^{\bullet} \shH$ and their properties are what matters most for
algebraic geometers.

Variations of Hodge structure are ``smooth'' objects, arising for example from
families of smooth projective varieties. One can include objects with singularities
by generalizing vector bundles with flat connection to $\Dmod$-modules. In fact,
$(\shH, \nabla)$ is a special case of a regular holonomic $\Dmod$-module. The
connection gives rise to an action by tangent vector fields (= linear differential
operators of order one), according to the formula
\[
	\xi \cdot s = \nabla_{\xi}(s) \quad 
		\text{for $\xi \in \Gamma(U, \shT_X)$ and $s \in \Gamma(U, \shH)$}.
\]
This makes $\shH$ into a left $\Dmod$-module: on the one hand, $\nabla$ satisfies the
Leibniz rule
\[
	\nabla_{\xi}(fs) = (\xi f) s + f \, \nabla_{\xi}(s), 
\]
which amounts to the relation $\lie{\xi}{f} = \xi f$; on the other hand, $\nabla$
flat means that
\[
	\nabla_{\xi} \bigl( \nabla_{\eta}(s) \bigr) 
		- \nabla_{\eta} \bigl( \nabla_{\xi}(s) \bigr) = \nabla_{\lie{\xi}{\eta}}(s),
\]
which amounts to the relation $\xi \eta - \eta \xi = \lie{\xi}{\eta}$. This gives
$\shH$ the structure of a left module over the sheaf of linear differential
operators $\Dmod_X$. The Hodge filtration defines a filtration of $\shH$ that is
compatible with the order of differential operators. Indeed, if we set $F_k \shH =
F^{-k} \shH$ to get an increasing filtration, then we have
\[
	\shT_X \cdot F_k \shH \subseteq F_{k+1} \shH
\]
because of Griffiths transversality. The resulting $\Dmod$-module is regular
holonomic; its so-called characteristic variety, a subset of the cotangent bundle
$T^{\ast} X$, is precisely the zero section. This corresponds to the fact that,
locally on $X$, solutions to the equation $\nabla s = 0$ can be analytically
continued in every direction.

\subsection{Hodge modules}

Now let $X$ be a smooth quasi-projective variety; Hodge modules can be defined much
more generally, including on singular varieties, but we shall focus on this case
because it simplifies the discussion. The essential features of a Hodge module are
then the following:
\begin{enumerate}
\item A regular holonomic $\Dmod_X$-module $\Mmod$;
\item a compatible filtration $F_{\bullet} \Mmod$ by coherent $\OX$-submodules that
is ``good'', meaning that $\shT_X \cdot F_k \Mmod \subseteq F_{k+1} \Mmod$, with
equality for $k \gg 0$. 
\end{enumerate}
As in the case of variations of Hodge structure, there are several additional pieces
of data, such as a polarization; a rational structure (in the form of a perverse sheaf
with coefficients in $\QQ$); a weight filtration (in the mixed case); etc. There is
also a long list of axioms that need to be satisfied in order for the pair $(\Mmod,
F_{\bullet} \Mmod)$ to be called a Hodge module.

A basic fact is that every variation of Hodge structure defines a Hodge module; we
saw already more or less how this works. In the other direction, every Hodge
module is generically a variation of Hodge structure. 

\begin{example}
In the situation of \theoremref{thm:main}, there is a Zariski-open subset over which
$M$ is a variation of mixed Hodge structure. Under the indexing conventions explained
above, the condition $F_{k-1} \Mmod = 0$ translates into $F^{-k+1} \shH = 0$, and the
coherent sheaf $F_k \Mmod$ in the theorem is an extension of the Hodge bundle $F^{-k}
\shH$. Note that the theorem allows the possibility that $\Supp \Mmod \neq X$: in
that case, $\shH = 0$ and the assertion about weak positivity becomes trivial.
\end{example}

\subsection{Hyperplane sections and residues}

Instead of boring the reader with a list of axioms, let me motivate the use of
$\Dmod$-modules by an example. The example is the residue description for the
cohomology of hyperplane sections, something that I though about a lot for my
dissertation. This is a good place to acknowledge the great influence that Herb has
had on my mathematical interests: I basically learned about mixed Hodge modules by
trying to understand some of his constructions with residues and differential
operators.

Let $X$ be a smooth projective variety of dimension $n$, and let $Y \subseteq X$ be a
smooth and very ample divisor in $X$. According to the Lefschetz theorems, the
cohomology groups $H^k(Y) = H^k(Y, \CC)$ are determined by those of $X$, with the
exception of the so-called \define{variable part}
\[
	\Hev^{n-1}(Y) = \ker \bigl( H^{n-1}(Y) \to H^{n+1}(X)(1) \bigr),
\]
defined as the kernel of the Gysin morphism. The variable part can be described very
nicely by residues. We denote by $\OmX^k(\ast Y)$ the sheaf of meromorphic $k$-forms
on $X$ that are holomorphic on $X \setminus Y$, but may have poles along $Y$. By
Grothendieck's theorem, the hypercohomology of the algebraic de Rham complex
\[
	\begin{tikzcd}[column sep=small]
	\Bigl\lbrack \OX(\ast Y) \rar{d} & \OmX^1(\ast Y) \rar{d} &
		\dotsb \rar{d} & \OmX^n(\ast Y) \Bigr\rbrack
	\end{tikzcd}
\]
computes $H^{\ast}(X \setminus Y)$; this is a consequence of the fact that $X
\setminus Y$ is affine. From the long exact cohomology sequence
\[
	\dotsb \to H^{n-2}(Y)(-1) \to H^n(X) \to H^n(X \setminus Y) \to H^{n-1}(Y)(-1)
		\to H^{n+1}(X) \to \dotsb
\]
we obtain a short exact sequence
\[
\begin{tikzcd}[column sep=small]
0 \rar & H_0^n(X) \rar & H^n(X \setminus Y) \rar & \Hev^{n-1}(Y)(-1) \rar & 0 \\
& & H^0 \bigl( X, \OmX^n(\ast Y) \bigr) \uar 
	\arrow[swap,bend right=30,start anchor={[yshift=0.8ex]}]{ur}{\Res_Y}
\end{tikzcd}
\]
The arrow labeled $\Res_Y$ is the so-called \emph{residue mapping}; under our assumptions on
$Y$, it is surjective. Carlson and Griffiths \cite{CarlsonGriffiths} have shown that,
if the line bundle $\OX(Y)$ is sufficiently ample, then the Hodge filtration on
$\Hev^{n-1}(Y)$ is the filtration by pole order. More precisely, their result is that
\[
	\Res_Y \colon H^0 \bigl( X, \OmX^n(kY) \bigr) \to F^{n-k} \Hev^{n-1}(Y)
\]
is surjective. For $k=1$, this recovers the familiar fact that $\Hev^{n-1,0}(Y)$ is
generated by residues of logarithmic $n$-forms on $X$.
	
The result of Carlson and Griffiths also works in families. Let $L$ be a very ample
line bundle on $X$, and let $P = \lvert L \rvert$ denote the linear system of its
sections. Each point $p \in P$ corresponds to a hypersurface $Y_p \subseteq X$; we
denote by $P_0 \subseteq P$ the set of points where $Y_p$ is smooth, and by $j \colon
P_0 \into P$ the inclusion. Let
\[
	\famY = \menge{(p, x) \in P \times X}{x \in Y_p} \subseteq P \times X
\]
be the incidence variety. Just as above, we can use residues to express the
variable part in the variation of Hodge structure $R^{n-1} p_{1\ast}
\QQ_{\famY}$; more precisely, what we get is a description of the underlying flat
vector bundle $(\shHev, \nabla)$ and the Hodge filtration $F^{\bullet} \shHev$ on it.

\subsection{Extension to singular hyperplane sections}

It is an interesting question whether the description from above can also tell us
something about singular hypersurfaces. The answer is that we can use residues to
construct a natural extension of the bundle $\shH$ from $P_0$ to all of $P$; this idea is
due to Herb Clemens. Suppose for a moment that we have a relative $n$-form 
\[
	\omega \in \Gamma \bigl( U \times X, \Omega_{P \times X/P}^n(\ast \famY) \bigr)
\]
with poles along $\famY$, defined over an open subset of the form $U \times X$. At
each point $p \in U \cap P_0$, the hypersurface $Y_p$ is smooth, and so we have a
well-defined residue
\[
	\Res_{Y_p} \bigl( \omega \restr{\{p\} \times X} \bigr) \in \Hev^{n-1}(Y_p).
\]
In this way, we obtain a holomorphic section of the bundle $\shHev$ over the open set
$U \subseteq P$; for simplicity, we shall denote it by the symbol $\Res(\omega)$. We
can then define a subsheaf $\Mmod \subseteq \jl \shH$ by the following rule: its
sections over an open subset $U \subseteq P$ are given by
\[
	\Gamma(U, \Mmod) = \Menge{s \in \Gamma(U \cap P_0, \shH)}{%
		\text{$s = \Res(\omega)$ for some $\omega \in \Gamma \bigl( U \times X, 
			\Omega_{P \times X/P}^n(\ast \famY)$}}.
\]
It has a natural filtration $F_{\bullet} \Mmod$ by pole order, defined as follows:
\[
	\Gamma(U, F_k \Mmod) = \Menge{s \in \Gamma(U, \Mmod)}{%
		\text{$s = \Res(\omega)$ for some $\omega \in \Gamma \bigl( U \times X, 
			\Omega_{P \times X/P}^n(k \famY)$}}.
\]
The point is that we consider only those sections of $\jl \shH$ that are the residue
of a meromorphic form with poles along $\famY$; note that the meromorphic form needs
to be defined on all of $U$, including over points of $P \setminus P_0$. This makes
sense because the incidence variety $\famY$ is actually a smooth hypersurface in $P
\times X$.

Provided that the line bundle $L$ is sufficiently ample, the theorem of Carlson and
Griffiths from above shows that we have
\[
	F_k \Mmod \restr{P_0} \simeq F_{k-n} \shHev = F^{n-k} \shHev
		\quad \text{and} \quad
	\Mmod \restr{P_0} \simeq \shHev.
\]

Now $\Mmod$ is naturally a left $\Dmod_P$-module: differential
operators on $P$ act by differentiating the coefficients of $\omega$. More precisely,
given a vector field $\xi \in \Gamma(U, \shT_P)$, let $\xit \in \Gamma(U \times X,
\shT_{P \times X})$ denote the obvious lifting to the product; then if $s =
\Res(\omega)$, we can take the Lie derivative and define
\[
	\xi \cdot s = \Res \bigl( \mathcal{L}_{\xit} \omega \bigr).
\]
Concretely, this means that we contract $d \omega$ with $\xit$, and then take the
residue of the resulting relative $n$-form. It is easy to see that the filtration
$F_{\bullet} \Mmod$ is compatible with the order of differential operators -- after
all, differentiating $k$ times will increase the order of the pole by $k$. One can
also compute the characteristic variety $\Ch(\Mmod)$ and prove in this way that
$\Mmod$ is holonomic. The characteristic variety turns out to be closely related to
the set 
\[
	\famY' = \menge{(p, x) \in \famY}{\text{$Y_p$ is singular at $x$}}
\]
of singular points in the fibers of $p_1 \colon \famY \to P$; this is not surprising,
because those are exactly the points where one cannot take a residue in the classical
sense. Both $\famY$ and $\famY'$ are projective bundles over $X$, of
dimension $\dim X + \dim P - 1$ and $\dim P - 1$, respectively; they naturally embed
into the projectivization of $T^{\ast} P$. 

\begin{lemma}
If $\shHev \neq 0$, then the characteristic variety of $\Mmod$ is given by
\[
	\Ch(\Mmod) = \bigl( \text{the zero section of $T^{\ast} P$} \bigr) \cup
		\bigl( \text{the cone over $\famY'$} \bigr).
\]
In particular, both components of $\Ch(\Mmod)$ are of dimension $\dim P$, which
means that $\Mmod$ is holonomic.
\end{lemma}

One can show that $(\Mmod, F_{\bullet + n} \Mmod)$ is part of a Hodge module on $P$
that naturally extends the variation of Hodge structure $\shHev$. You can find the
details in the paper \cite{Schnell-R}; the key point in the proof is that the
incidence variety $\famY$ is smooth. 

\subsection{Conclusion}

The message to take away from this is that Hodge modules are the correct
generalization of variations of Hode structure. Indeed, here we have one example
where we get a natural extension of a variation of Hodge structure (with the help of
residues) \dots\ and it turns out that the extension is precisely the same as the one
given by Saito's theory! I should add that I have found this particular example to be
very useful for learning about mixed Hodge modules.

\section{Proof of the main theorem}

The remainder of the paper is devoted to the proof of \theoremref{thm:main}. The
proof closely follows Koll\'ar's, but we replace certain geometric arguments by
abstract results about mixed Hodge modules. As I have said before, I do not claim that
this proof is in any
way simpler than the original one; only that it is shorter, and therefore perhaps
better suited for generalizations.

\subsection{Saito's vanishing theorem}

The main ingredient is the following vanishing theorem for the first nonzero sheaf
in the Hodge filtration of a mixed Hodge module \cite{Saito-K}*{Theorem~1.2}.

\begin{theorem}[Saito] \label{thm:Saito}
Let $(\Mmod, F_{\bullet})$ be the filtered $\Dmod$-module underlying a mixed Hodge
module on a smooth projective variety $X$. If $F_{k-1} \Mmod = 0$, then 
\[
	H^i \bigl( X, \omX \tensor L \tensor F_k \Mmod \bigr) = 0
\]
for every $i > 0$ and every ample line bundle $L$ on $X$.
\end{theorem}

The vanishing theorem implies that $\omX \tensor L^{\tensor(n+1)} \tensor F_k \Mmod$
is $0$-regular (in the sense of Castelnuovo-Mumford), and therefore always globally
generated (here and below, $n = \dim X$). To prove that $F_k \Mmod$ is weakly
positive, we have to find a way to raise the term $F_k \Mmod$ in the formula to a
large power. In Koll\'ar's proof, this is accomplished by considering the fiber
product $X \times_Y \dotsm \times_Y X \to Y$; it requires some analysis of the
singularities. We shall replace this geometric argument by the following abstract
result about mixed Hodge modules.

\subsection{A useful lemma}

Next, we recall a useful lemma about restriction to submanifolds; it appears in
\cite{Schnell-N}*{Lemma~2.17}. The symbol $\ius M$ in the statement refers to a
certain pullback operation on mixed Hodge modules (in the derived category); the
precise definition does not matter for our purposes.

\begin{lemma} \label{lem:ius}
Let $M$ be a mixed Hodge module on a smooth complex algebraic variety $X$, and let
$(\Mmod, F)$ denote the underlying filtered $\Dmod_X$-module.  Let $i \colon Z \into
X$ be the inclusion of a smooth subvariety of codimension $r$, and let $\ius(\Mmod,
F)$ denote the complex of filtered $\Dmod_Z$-modules underlying the object $\ius M
\in \Db \MHM(Z)$. 
\begin{aenumerate}
\item \label{en:ius-a}
For every $p \in \ZZ$, there is a canonical morphism
\[
	F_{k-r} \, \ius(\Mmod, F) \to \derL \iu (F_k \Mmod) \decal{-r}
\]
in the derived category $\Dbcoh(\OZ)$. 
\item \label{en:ius-b}
The morphism in \ref{en:ius-a} is an isomorphism over the locus where $M$ is smooth.
\end{aenumerate}
\end{lemma}

\begin{proof}
The functor $\ius$ is the right adjoint of $\il$, and since $i$ is a closed
embedding, we have the adjunction morphism $\il \ius M \to M$ in $D^b \MHM(X)$.
Passing to filtered $\Dmod$-modules, we then get a morphism $F_k \bigl( \il
\ius(\Mmod, F) \bigr) \to F_k \Mmod$. Now $i$ is a closed embedding of codimension
$r$, and because of how the direct image functor is defined in \cite{Saito-HM}*{2.3},
we have a canonical morphism
\[
	\il \bigl( \omega_Z \tensor F_{k-r}\, \ius(\Mmod, F) \bigr)
		\to \omega_X \tensor F_k \bigl( \il \ius(\Mmod, F) \bigr).
\]
Let $\omega_{Z/X} = \omega_Z \tensor \iu \omega_X^{-1}$. Composing the two morphisms,
we obtain
\[
	\il \bigl( \omega_{Z/X} \tensor F_{k-r} \, \ius(\Mmod, F) \bigr)
		\to F_k \Mmod.
\]
On the level of coherent sheaves, the functor $\derL \ius = \omega_{Z/X} \decal{-r}
\tensor \derL \iu$ is the right adjoint of $\il$; by adjunction, we therefore get the
desired morphism 
\[
	F_{k-r} \, \ius(\Mmod, F) \to \omega_{Z/X}^{-1} \tensor \derL \ius(F_k \Mmod)
		= \derL \iu(F_k \Mmod) \decal{-r}
\]
in the derived category $\Dbcoh(\shO_Z)$. 

Now let us prove \ref{en:ius-b}. After restricting to the open subset where $M$ is
smooth, we may assume that $M$ is the mixed Hodge module associated with a variation
of mixed Hodge structure; in particular, all the sheaves $F_k \Mmod$, as well as
$\Mmod$ itself, are locally free. In this situation, the complex $\ius M$ is
concentrated in degree $r$, and $H^r \ius M$ is simply the restriction of $M(r)$ to
$Z$. It follows that
\[
	F_{k-r} \, \ius(\Mmod, F) = \iu(F_k \Mmod) \decal{-r}.
\]
On the other hand, $\derL \iu(F_k \Mmod) = \iu(F_k \Mmod)$ because $F_k \Mmod$ is
locally free. It is then easy to see from the construction that the morphism in
\ref{en:ius-a} is an isomorphism.
\end{proof}

\subsection{The proof}

Now consider a mixed Hodge module $M$ as in \theoremref{thm:main}. Let $(\Mmod,
F_{\bullet} \Mmod)$ denote the underlying filtered $\Dmod$-module; without loss of
generality, we may assume that $\Supp \Mmod = X$. We apply \lemmaref{lem:ius} to the
diagonal embedding
\[
	\Delta \colon X \into X \times \dotsm \times X;
\]
if there are $m$ factors, its codimension is $r = (m-1)n$. On $X^m$, we have a mixed
Hodge module
\[
	\Mt = M \boxtimes \dotsm \boxtimes M.
\]
The Hodge filtration on the underlying $\Dmod$-module $\Mmodt$ is defined by
convolving the filtrations on the individual factors; in particular,
\[
	F_{mk-1} \Mmodt = 0 \quad \text{and} \quad
		F_{mk} \Mmodt = (F_k \Mmod)^{\boxtimes m}.
\]
\lemmaref{lem:ius} thus gives us a canonical morphism
\[
	F_{mk-r} \Delta^{!} \bigl( \Mmodt, F \bigr)
		\to \derL \Delta^{\ast} \bigl( F_{mk} \Mmodt \bigr) \decal{-r}
\]
in the derived category. If we take cohomology in degree $r$, we obtain
\[
	F_{mk-r} \, \Nmod \to (F_k \Mmod)^{\tensor m},
\]
where $\Nmod$ is the $\Dmod$-module underlying a certain mixed Hodge module on $X$.
The important thing is that this morphism is an isomorphism over the open set $U
\subseteq X$ where $M$ is a variation of mixed Hodge structure. In fact, the proof of
\lemmaref{lem:ius} shows that, over $U$, both sides are just isomorphic to the $m$-th
tensor power of the Hodge bundle $F^{-k} \shH$.

Now we can easily show that $F_k \Mmod$ is weakly positive on $U$. Fix a very ample
line bundle $L$ on $X$. \theoremref{thm:Saito} implies that the sheaf 
\[
	\omega_X \tensor L^{\tensor(n+1)} \tensor F_{mk-r} \, \Nmod
\]
is $0$-regular and therefore globally generated. Because the morphism
\[
	\omega_X \tensor L^{\tensor(n+1)} \tensor F_{mk-r} \, \Nmod 
		\to \omega_X \tensor L^{\tensor(n+1)} \tensor (F_k \Mmod)^{\tensor m}
\]
is an isomorphism over $U$, it follows that the sheaf
\[
	\omega_X \tensor L^{\tensor(n+1)} \tensor (F_k \Mmod)^{\tensor m}
\]
is generated over $U$ by its global sections. Since $m$ was arbitrary, this implies
that $F_k \Mmod$ is weakly positive on $U$.

\section*{Acknowledgement}

I thank the other organizers, especially Gary Kennedy and Mirel Caib\u{a}r, for making
the conference a success. I am grateful to Mihnea Popa for many discussions about
Hodge modules and weak positivity, and to Mark Green for pointing out the correct
reference for the result by Carlson and Griffiths.

\bibliographystyle{amsplain}

\end{document}